\documentclass[11pt]{article}

\usepackage{amsmath,amsthm,amssymb,bm,array,enumerate,tikz-cd}
\usepackage[
top=3cm, bottom=3cm,
outer=3cm, inner=3cm,
headheight=0.4cm, headsep=0.4cm,
footskip=0cm]{geometry}

\usepackage{fancyhdr} 
\usepackage{titlesec}
\titleformat{\section}{\large\sc\center}{\thesection.}{0.5em}{} 
\titleformat{\subsection}[runin]{\normalsize\bf}{\thesubsection.}{0.4em}{\addperiod}
\titleformat{\subsubsection}[runin]{\normalsize\it}{\thesubsubsection.}{0.4em}{\addperiod}
\newcommand{\addperiod}[1]{#1.}
\titlespacing*{\section}{0pt}{1.3\baselineskip}{0.2\baselineskip}
\titlespacing*{\subsection}{0pt}{0.1\baselineskip}{.3\parindent}
\titlespacing*{\subsubsection}{0pt}{-0.1\baselineskip}{.2\parindent}
\renewcommand{\thesection}{\arabic{section}}
\renewcommand{\thesubsection}{\arabic{section}.\arabic{subsection}}
\renewcommand{\thesubsubsection}{\arabic{section}.\arabic{subsection}.\arabic{subsubsection}}

\usepackage{cite}

\renewcommand{\theequation}{\arabic{section}.\arabic{equation}}
\newtheorem{thm}{Theorem}[section]
\newtheorem{lem}[thm]{Lemma}
\newtheorem{prop}[thm]{Proposition}
\newtheorem{cor}[thm]{Corollary}
\newtheorem{rmk}[thm]{Remark}

\newtheorem{thmalph}{Theorem}

\newtheorem{coralph}[thmalph]{Corollary}

\newtheorem*{thm*}{Theorem}
\newtheorem*{ac*}{Acknowledgement}
\newtheorem*{clm*}{Claim}

\renewcommand{\eqref}[1]{{\upshape(\ref{#1})}}

\def\C{\mathbb{C}}
\def\F{\mathbb{F}}
\def\H{\mathbb{H}}

\def\N{\mathbb{N}}\def\P{\mathbb{P}}
\def\R{\mathbb{R}}

\def\bO{\mathbb{O}}
\def\CC{\mathcal{C}}
\def\FF{\mathcal{F}}

\def\NN{\mathcal{N}}

\def\OO{\mathcal{O}}

\def\cS{\mathcal{S}}% "\SS" is chapter symbol

\usepackage[scr=rsfso]{mathalfa}
\def\fra{\mathfrak{a}}

\def\frg{\mathfrak{g}}\def\frn{\mathfrak{n}}

\def\Ga{\Gamma}
\def\la{\lambda}

\def\one{\bm{1}}
\def\zero{\bm{0}}

\def\({\big(}\def\){\big)}% \(....\)

\def\Spin{\textup{Spin}}
\def\upU{\textup{U}}
\def\GL{\textup{GL}}

\def\ad{\textup{ad}}

\def\Ind{\textup{Ind}}\def\Hom{\textup{Hom}}

\def\vol{\textup{vol}}
\renewcommand\Re{\operatorname{Re}}\renewcommand\Im{\operatorname{Im}}
\DeclareMathOperator{\supp}{supp}

\DeclareMathOperator{\diag}{diag}

\DeclareMathOperator{\Isom}{Isom}
\DeclareMathOperator{\Tr}{Tr}
\newcommand\const{\operatorname{const}}

\def\aut{\textup{aut}}\renewcommand{\mod}{\textup{mod}}

\usepackage[hidelinks,bookmarks=false]{hyperref}

\begin{document}

\title{Upper bounds for geodesic periods over rank one locally symmetric spaces}

\author{Jan M\"ollers
\and
Feng Su}

\date{}

\maketitle

\thispagestyle{empty} 

\allowdisplaybreaks

\begin{abstract}
We prove upper bounds for geodesic periods of automorphic forms over general rank one locally symmetric spaces. Such periods are integrals of automorphic forms restricted to special totally geodesic cycles of the ambient manifold and twisted with automorphic forms on the cycles. The upper bounds are in terms of the Laplace eigenvalues of the two automorphic forms, and they generalize previous results for real hyperbolic manifolds to the context of all rank one locally symmetric spaces.
\end{abstract}

\section*{Introduction}

Estimating the restriction of a Laplace eigenfunction $f$ on a Riemannian manifold $X$ to a compact submanifold $Y$ is a classical problem in partial differential equations and global analysis. There are various types of restriction problems, for example, one may estimate $L^p$-norms ($0<p\leq\infty$) of the restriction $f|_Y$ in terms of the eigenvalue of $f$. A more refined quantity in the case $p=2$ are the Fourier coefficients of $f|_Y$ with respect to an orthonormal basis of Laplace eigenfunctions $g$ in $L^2(Y)$. Each Fourier coefficient is given by an integral $P_Y(f,g)$ of $f|_Y$ against $g$, and one can ask for estimates in terms of the eigenvalues of $f$ and $g$.

In the context of locally symmetric spaces, the integrals $P_Y(f,g)$ are also called \textit{periods} and they carry important arithmetic information of $f$ and $g$. One instance of this is when $X$ is an arithmetic hyperbolic surface and $Y$ a closed geodesic, in which case the period integrals $P_Y(f,g)$ are proportional to special $L$-values via Waldspurger's formula. In a series of papers, J.~Bernstein and A.~Reznikov~\cite{BR04,BR05,BR10,Rez08,Rez15} studied various types of period integrals for automorphic forms on hyperbolic surfaces in connection with the representation theory of the corresponding isometry groups of the universal coverings of $X$ and $Y$. More recently, their techniques were generalized to the case of higher-dimensional real hyperbolic manifolds by the authors~\cite{MO13,Su16}, where estimates for $P_Y(f,g)$ either in terms of the eigenvalue of $f$ or in terms of the eigenvalue of $g$ were obtained.

In this work we try to extend the techniques by Bernstein--Reznikov even further and study period integrals for arbitrary locally symmetric spaces $X$ of rank one. In this context, we obtain estimates for $P_Y(f,g)$ in terms of the eigenvalue of $f$ and $g$, for compact totally geodesic cycles $Y$ of a particular form. This includes real, complex and quaternionic hyperbolic manifolds $X$ as well as quotients of the $16$-dimensional octonionic plane in which case $Y$ is an $8$-dimensional real hyperbolic submanifold.

\subsection*{Statement of the results}

Let $X$ be a connected locally symmetric space of rank one. Then the universal cover $\widetilde{X}$ of $X$ is isomorphic to the hyperbolic space $H_\F^n$ over $\F$, where $\F=\R,\C,\H$ ($n\geq2$) or $\F=\bO$ ($n=2$). The real dimension of $X$ is equal to $N=nd$, where $d=\dim_\R\F$. We can write $X=\Gamma\backslash H^n_\F$ for some discrete torsion-free subgroup $\Gamma\subseteq\Isom(H^n_\F)$ of the isometry group of $H^n_\F$. The space $X$ is a complete Riemannian manifold and we denote by $dx$ the Riemannian volume element of $X$ and by $\Delta_X$ the Laplacian on $X$. Then $\Delta_X$ is a self-adjoint operator on $L^2(X)=L^2(X,dx)$.

Consider a compact totally geodesic cycle $Y\subseteq X$ of the form $Y=(\Gamma\cap\Isom(H^m_\F))\backslash H^m_\F$ ($1\leq m<n$), where $H^m_\F$ is naturally embedded into $H^n_\F$. Write $M=md$ for the real dimension of $Y$, $dy$ for its Riemannian volume element and $\Delta_Y$ for the Laplacian on $Y$.

Let $f$ be a normalized square integrable Laplace eigenfunction on $X$ with eigenvalue $a\geq0$, i.e. $\Delta_Xf=af$. Then $f$ is smooth by elliptic regularity. Integrating the restriction $f|_Y$ of $f$ to $Y$ against a normalized eigenfunction $g\in L^2(Y)$ of $\Delta_Y$ with eigenvalue $b\geq0$, i.e. $\Delta_Yg=bg$, defines a period integral
$$ P_Y(f,g) = \int_Y f(y)g(y)\,dy,$$
called \emph{geodesic period}. Since $Y$ is compact, the integral always converges absolutely. We are interested in the asymptotic behavior of $P_Y(f,g)$ as the eigenvalues of $f$ and $g$ grow. Let $\{g_j\}_j$ be an orthonormal basis of $L^2(Y)$ consisting of Laplace eigenfunctions with eigenvalues $0=b_0<b_1\leq b_2\leq\ldots$. Then the period integrals $P_Y(f,g_j)$ are the Fourier coefficients of the expansion of $f|_Y$ into the basis $\{\bar{g}_j\}$ of $L^2(Y)$.

To simplify the statement of our main result, define the modified period
$$ \P_Y(f,g) = |a-b|^{-\frac{N-M-2}{4}}e^{\frac{\pi}{2}\sqrt{|a-b|}}P_Y(f,g). $$

We shall prove the following:

\begin{thmalph}\label{thm:Main}
There exists a constant $C=C_\Gamma>0$ such that
$$ \sum_{b_j\leq T} |\P_Y(f,g_j)|^2 \leq C\cdot a^{-\frac{N-2}{2}}e^{\pi\sqrt{a}}\cdot T^{\frac{N-1}{2}} \qquad \forall\,f\mbox{ with }a\leq T. $$
\end{thmalph}

For the case where one of the two eigenfunctions $f$ and $g$ is fixed, Theorem~\ref{thm:Main} implies the following bounds for the original geodesic periods $P_Y(f,g)$:

\begin{coralph}\label{cor:Main}
\begin{enumerate}[(1)]
\item\label{cor:Main1} For a fixed Laplace eigenfunction $f\in L^2(X)$ there exists a constant $C=C_{\Gamma,f}>0$ such that
$$ |P_Y(f,g_j)| \leq C\cdot b_j^{\frac{2N-M-3}{4}}e^{-\frac{\pi}{2}\sqrt{b_j}} \qquad \mbox{as $j\to\infty$.} $$
\item\label{cor:Main2} For a fixed Laplace eigenfunction $g\in L^2(Y)$ there exists a constant $C=C_{\Gamma,g}>0$ such that
$$ |P_Y(f,g)| \leq C\cdot a^{\frac{N-M-1}{4}} \qquad \forall\,f. $$
\end{enumerate}
\end{coralph}

We remark that our results hold for all locally symmetric spaces of rank one (not necessarily compact), in particular also for quotients $X$ of the octonionic plane $H^2_\bO$ in which case $Y$ is a compact $8$-dimensional hyperbolic manifold.

\subsection*{Relation to previous work}

Let us compare our results to the existing literature:
\begin{itemize}
\item 
Burq--G\'{e}rard--Tzvetkov~\cite{BGT07}
estimated the $L^p$-norm ($2\leq p\leq\infty$) of $f|_Y$ for general compact Riemannian manifolds $X$ and compact submanifolds $Y$. The resulting bound for a single period is weaker than Corollary~\ref{cor:Main}~\eqref{cor:Main2}, which is to be expeceted as our results are for a very special family of Riemannian manifolds.
\item
Zelditch \cite{Ze92} estimated asymptotics of the period $P_Y(f,1)$ (i.e. $g$ is a constant function) where $X$ is a compact Riemannian manifold and $Y$ is any compact submanifold. The resulting bound for $P_Y(f,1)$ is the same as in Corollary~\ref{cor:Main}~\eqref{cor:Main2}.
\item 
M\"{o}llers--{\O}rsted~\cite{MO13} estimated the period $P_Y(f,g)$ where $X$ is a real hyperbolic manifold (not necessarily compact) and $Y$ a special totally geodesic cycle of any dimension. In their work the form $f$ is fixed and $g$ varies, and the obtained bound is a special case of Corollary~\ref{cor:Main}~\eqref{cor:Main1}.
\item 
Su~\cite{Su16} estimated the period $P_Y(f,g)$ where $X$ is a hyperbolic manifold (not necessarily compact) and $Y$ a special totally geodesic cycle of codimension 1. Here the form $f$ varies and $g$ is fixed, and the obtained bound is a special case of Corollary~\ref{cor:Main}~\eqref{cor:Main2}.
\end{itemize}

\subsection*{Method of proof}

Each automorphic form $f\in L^2(\Gamma\backslash H^n_\F)$ corresponds to an irreducible unitarizable representation $\pi_f$ of the isometry group $G=\upU(1,n;\F)$ of the hyperbolic space $H^n_\F$. Then $H=\upU(1,m;\F)\times\upU(n-m;\F)\subseteq G$ is the subgroup of isometries leaving the hyperbolic subspace $H^m_F\subseteq H^n_\F$ invariant, and automorphic forms $g\in L^2((\Gamma\cap H)\backslash H^m_\F)$ correspond to certain irreducible unitarizable representations $\tau_g$ of $H$. The period integral $P_Y(f,g)$ can be interpreted as the special value of an $H$-invariant bilinear form on the product of $\pi_f$ and $\tau_g$. In this context, invariant bilinear forms are unique up to scalars and can therefore be related to invariant bilinear forms on explicit models of the representations $\pi_f$ and $\tau_g$. Such model forms are constructed in \cite{MOO16} for principal series representations in terms of explicit integral kernels, so that the period integrals $P_Y(f,g)$ can be written as the product of a certain explicit model integral and a proportionality scalar relating the two invariant bilinear forms (an idea due to Bernstein--Reznikov). Whereas the proportionality scalar can be estimated by standard techniques (see Section~\ref{estimate}), the evaluation of the model integral is the key computation in this paper (see Section~\ref{special}).

\subsection*{Structure of the paper}

In Section \ref{sec:GeodesicPeriods} we give a group-theoretic description of locally symmetric spaces of rank one and their special totally geodesic cycles. Also we introduce the automorphic invariant bilinear forms arising from geodesic periods. In Section \ref{representation} we recall the structure and representation theory of rank one groups. The non-compact models of the representations and invariant bilinear forms on these models are discussed. The geodesic periods are expressed as a product of a special value of the model invariant bilinear form and a proportionality scalar. The latter two objects are dealt with in Section \ref{special} and \ref{estimate}, respectively. The necessary integral formulas used in Section \ref{special} are summarized in the Appendix.\newpage

\section{Geodesic periods on rank one locally symmetric spaces}\label{sec:GeodesicPeriods}

\subsection{Locally symmetric spaces of rank one}\label{geometry}

Let $X$ be a connected locally symmetric space of rank one. Then its universal covering $\widetilde{X}$ is a globally symmetric space of rank one and hence of the form $\widetilde{X}=G/K$, where $G=\upU(1,n;\F)$ with $\F=\R,\C,\H$ ($n\geq2$) or $\F=\bO$ ($n=2$) and $K=\upU(1;\F)\times\upU(n;\F)$ the standard maximal compact subgroup of $G$. The group $G$ acts on $\widetilde{X}$ by isometries, and conversely every isometry of $\widetilde{X}$ is given by an element of $G$. We can therefore identify the fundamental group $\Gamma=\pi_1(X)$ of $X$ with a torsion-free discrete subgroup of $G$ so that $X=\Gamma\backslash G/K$.

The symmetric space $\widetilde{X}$ is a Riemannian manifold with the metric induced by the Killing form $\kappa$ on the Lie algebra $\frg$ of $G$ which we normalize to be
$$ \kappa(X,Y) = \frac{1}{2}\Tr(XY^*), \qquad X,Y\in\frg.$$
This metric descends to $X$ and gives rise to a Riemannian measure $dx$ of $X$. Denote by $\Delta_X$ the Laplacian on $X$ determined by the Riemannian metric. Then $\Delta_X$ extends to a self-adjoint operator on $L^2(X)$.

\subsection{Automorphic representations}

Every eigenfunction $f\in L^2(X)$ of $\Delta_X$ is automatically smooth since $\Delta_X$ is an elliptic operator. In view of the isomorphism $L^2(X)\simeq L^2(\Gamma\backslash G)^K$ we can regard $f$ as a right-$K$-invariant $L^2$-function on $\Gamma\backslash G$. The space $\Gamma\backslash G$ is equipped with a $G$-invariant Radon measure that descends to $dx$ (so that the quotient integral formula holds). Thus the action of $G$ given by right translation induces a unitary representation of $G$ on $L^2(\Gamma\backslash G)$. The closed subspace generated by all translates of $f$ is an irreducible subrepresentation of $L^2(\Gamma\backslash G)$, in which $f$ is the (up to scalar multiples) unique $K$-invariant vector.

Write $V_f\subseteq L^2(\Gamma\backslash G)$ for the space of smooth vectors of the subrepresentation generated by $f$, and $\pi_f$ for the action of $G$ on $V_f$. Since the Lie algebra $\mathfrak{g}$ of $G$ acts on $\Gamma\backslash G$ by vector fields, we have $V_f\subseteq C^\infty(\Gamma\backslash G)$ by the Sobolev Embedding Theorem.

\subsection{Geodesic cycles}

For $G=\upU(1,n;\F)$ we consider subgroups of the form $H=\upU(1,m;\F)\times\upU(n-m;\F)$ ($1\leq m\leq n-1$). Note that for $G=\upU(1,2;\bO)=F_{4(-20)}$ we have $H=\Spin(8,1)$. Then $(G,H)$ forms a symmetric pair. More precisely, $H=G^\sigma$ is the fixed point subgroup of the involution $\sigma$ of $G$ given by conjugation with the element $\diag(\one_{m+1},-\one_{n-m})$. The intersection $K_H=K\cap H=\upU(1;\F)\times\upU(m;\F)\times\upU(n-m;\F)$ is a maximal compact subgroup of $H$ and $H/K_H$ a rank one Riemannian symmetric space. The natural embedding $H/K_H\hookrightarrow G/K$ identifies $H/K_H$ with a totally geodesic submanifold of $G/K$.

Assume that $\sigma\Gamma=\Gamma$, then the intersection $\Gamma_H=\Gamma\cap H=\Gamma^\sigma\subseteq\Gamma$ is a torsion-free discrete subgroup of $H$ and the rank one locally symmetric space $Y=\Gamma_H\backslash H/K_H$ can be viewed as a totally geodesic submanifold of $X$. We denote by $dy$ the corresponding Riemannian measure on $Y$ and by $\Delta_Y$ the Laplacian.

Like in the case of $X$, every eigenfunction $g$ of $\Delta_Y$ generates an irreducible subrepresentation of $L^2(\Gamma_H\backslash H)$ whose subspace of smooth vectors is denoted by $W_g\subseteq C^\infty(\Gamma_H\backslash H)$. We write $\tau_g$ for the action of $H$ on $W_g$.

\subsection{Geodesic periods}

Let $f\in L^2(X)$ be a Maass form on $X$, i.e. $f$ is a square integrable eigenfunction of $\Delta_X$. Write $\Delta_Xf=af$ where $a\geq0$. Let $g\in L^2(Y)$ be a Maass form on $Y$ such that $\Delta_Yg=bg$ where $b\geq0$. Assuming that $Y$ is compact, the period integral
$$ P_Y(f,g) = \int_Yf(y)g(y)\,dy$$
converges absolutely. The purpose of this paper is to find estimates for $P_Y(f,g)$ in terms of $a$ and $b$. Throughout the paper we normalize $f$ and $g$ such that $\|f\|_{L^2(X)}=\|g\|_{L^2(Y)}=1$.

\subsection{Automorphic invariant bilinear forms}

We can view $P_Y(f,g)$ as the special value of an invariant bilinear form on $V_f\times W_g$. More precisely, define
$$ \ell^\aut_{f,g}:V_f\times W_g\to\CC, \quad (\phi,\psi)\mapsto\int_{\Gamma_H\backslash H}\phi(h)\psi(h)\,dh, $$
where $dh$ denotes the $H$-invariant measure on $\Gamma_H\backslash H$ that descends to $dy$, then
$$ P_Y(f,g) = \ell^\aut_{f,g}(f,g). $$
The integral above always converges since $\phi$, $\psi$ are smooth and $\Gamma_H\backslash H$ is compact. As a consequence of the $H$-invariance of $dh$, the bilinear map $\ell^\aut_{f,g}$ is invariant under the action $\pi_f|_H\otimes\tau_g$ of $H$, i.e.
$$ \ell^\aut_{f,g} \in \Hom_H(\pi_f|_H\otimes\tau_g,\C). $$

\section{Representation theory of rank one reductive groups}\label{representation}

We recall the classification of spherical unitary representations of the rank one groups $G=\upU(1,n;\F)$, $\F=\R,\C,\H,\bO$. For the subgroups $H=\upU(1,m;\F)\times\upU(n-m;\F)$ we describe the $H$-invariant bilinear forms on products of such representations of $G$ and $H$ as obtained in \cite{MOO16}.

\subsection{Group decompositions}

Let $\fra=\R H_0\subseteq\frg$ with
$$ H_0 = \left(\begin{array}{ccc}0&1&\\1&0&\\&&\zero_{n-1}\end{array}\right)\in\frg, $$
then $\ad(H_0)$ acts on $\frg$ with eigenvalues $\{0,\pm1\}$ for $\F=\R$, and eigenvalues $\{0,\pm1,\pm2\}$ for $\F=\C,\H,\bO$. Let $\alpha\in\fra_\C^*$ be given by $\alpha(H_0)=1$ and put
$$ \frn = \frg_\alpha+\frg_{2\alpha}, \qquad \overline{\frn} = \frg_{-\alpha}+\frg_{-2\alpha}. $$
Note that $\frg_{\pm2\alpha}=\{0\}$ for $\F=\R$. Denote by $A=\exp(\fra)$, $N=\exp(\frn)$ and $\overline{N}=\exp(\overline{\frn})$ the corresponding analytic subgroups of $G$. We identify $\F^{n-1}\oplus\Im\F\simeq\overline{N}$ by
$$ (x,z)\mapsto\overline{n}_{(x,z)} = \exp\left(\begin{array}{ccc}z&z&x^*\\-z&-z&-x^*\\x&x&\zero_{n-1}\end{array}\right)=\left(\begin{array}{ccc}1+z+\frac{1}{2}|x|^2&z+\frac{1}{2}|x|^2&x^*\\-z-\frac{1}{2}|x|^2&1-z-\frac{1}{2}|x|^2&-x^*\\x&x&\one_{n-1}\end{array}\right). $$
Then
$$ \overline{n}_{(x,z)}\cdot\overline{n}_{(x',z')}=\overline{n}_{(x,z)\cdot(x',z')}, $$
where
$$ (x,z)\cdot(x',z')= (x+x',z+z'+\tfrac{1}{2}(x^*x'-x^{\prime\,*}x)). $$
In particular, $\overline{n}_{(x,z)}^{-1}=\overline{n}_{(-x,-z)}$. Further, let
$$ M=Z_K(A)=\{\diag(z,z,k):z\in\upU(1;\F),k\in\upU(n-1;\F)\}\simeq\upU(1;\F)\times\upU(n-1;\F), $$
then $P=MAN$ is a (minimal) parabolic subgroup of $G$. Let $\rho=\frac{1}{2}\Tr\ad|_{\frn}$ be the half sum of positive roots.

We have the Iwasawa decomposition $G=KAN\simeq K\times A\times N$ and therefore, for every $g\in G$ there exists a unique $H(g)\in\fra$ such that $g\in Ke^{H(g)}N$. An easy computation shows that
$$ H(\overline{n}_{(x,z)}) = ((1+|x|^2)^2+4|z|^2)^{\frac{1}{2}}H_0, \qquad (x,z)\in\F^{n-1}\oplus\Im\F. $$

The subset $\overline{N}MAN\simeq\overline{N}\times M\times A\times N$ is open and dense in $G$, according to the Bruhat decomposition.

\subsection{The symmetric pair}

The parabolic subgroup $P$ is compatible with the subgroup $H$ in the sense that $P_H=P\cap H$ is a (minimal) parabolic subgroup of $H$. In fact, $P_H=M_HAN_H$ with $M_H=M\cap H$ and $N_H=N\cap H$. Let $\rho_H=\frac12\Tr\ad|_{\frn_H}$. The opposite nilradical $\overline{N}_H$ of $\overline{N}$ identifies with $\F^{m-1}\oplus\Im\F$ under the isomorphism $\F^{n-1}\oplus\Im\F\simeq\overline{N}$, where $\F^{m-1}\subseteq\F^{n-1}$ as the first $m-1$ coordinates.

\subsection{Principal series representations}

Identify $\fra_\C^*\simeq\C$ by $\lambda\mapsto\lambda(H)$ so that $\rho=\frac{1}{2}((n-1)d+2(d-1))=\frac{1}{2}(nd+d-2)$, where $d=\dim_\R\F$. For $\lambda\in\fra_\C^*$ we consider the principal series representations (smooth normalized parabolic induction)
$$ \pi_\lambda = \Ind_P^G(\one\otimes e^\lambda\otimes\one). $$
Since $\overline{N}MAN$ is open dense in $G$, any smooth right-$P$-equivariant function on $G$ is uniquely determined by its values on $\overline{N}$. We can therefore realize $\pi_\lambda$ on a space $I(\lambda)$ of smooth functions on $\overline{N}\simeq\F^{n-1}\oplus\Im\F$. More precisely, we have $\cS(\overline{N})\subseteq I(\lambda)\subseteq C^\infty(\overline{N})$ where $\cS(\overline{N})$ denotes the space of Schwarz functions on $\overline{N}$. The $K$-spherical vector in $I(\lambda)$ is given by the function
$$ \phi_\lambda(x,z) = ((1+|x|^2)^2+4|z|^2)^{-\frac{\lambda+\rho}{2}}, \qquad (x,z)\in\F^{n-1}\oplus\Im\F. $$

The representation $\pi_\lambda$ is irreducible and unitarizable if and only if $\lambda\in i\R\cup(-R,R)$, where the constant $R$ is given by
$$ R = \begin{cases}\rho&\mbox{for $\FF=\R,\C$,}\\\frac{1}{2}((n-1)d+2)&\mbox{for $\FF=\H,\mathbb{O}$.}\end{cases} $$
For $\lambda\in i\R$ the invariant inner product on $\pi_\lambda$ is the $L^2$-inner product of $L^2(\F^{n-1}\oplus\Im\F)$ with respect to the Lebesgue measure on $\F^{n-1}\oplus\Im\F\simeq\R^{(n-1)d+(d-1)}=\R^{nd-1}$. The irreducible unitary representations of $G$ with a non-zero $K$-spherical vector are precisely the unitary completions of $\pi_\lambda$ (where $\lambda\in i\R\cup(-R,R)$) together with the trivial representation. Note that $\pi_\lambda\simeq\pi_{-\lambda}$ for those parameters.

If now $f\in L^2(X)$ is an eigenfunction of the Laplacian $\Delta_X$, then $\pi_f$ is an irreducible unitarizable representation with $K$-spherical vector $f$, hence $\pi_f\simeq\pi_\lambda$ for some $\lambda\in i\R\cup(-R,R)$ or $\pi_f$ is the trivial representation. The parameter $\lambda$ is related to the eigenvalue $a$ of $f$ by
$$ a = \rho^2-\lambda^2. $$

For the subgroup $H$ we denote by $\tau_\nu$ ($\nu\in\fra_\C^*\simeq\C$) the corresponding principal series representation, realized on a space $J(\nu)$ of smooth functions on $\overline{N}_H\simeq\F^{m-1}\oplus\Im\F$ containing the spherical vector
$$ \psi_\nu(y,w) = ((1+|y|^2)^2+4|w|^2)^{-\frac{\nu+\rho_H}{2}}, \qquad (y,w)\in\F^{m-1}\oplus\Im\F, $$
where $\rho_H=\frac{1}{2}(md+d-2)$

\subsection{Model invariant bilinear forms}\label{sec:InvBilinearForms}

In \cite{MOO16} intertwining operators in $\Hom_H\big(\pi_\lambda|_H,\tau_\nu\big)$ were constructed and shown to be generically unique (see also \cite{KS15} for a much more thorough treatment of the special case where $\F=\R$ and $m=n-1$). We briefly review the results and show how they give rise to invariant bilinear forms on $I(\lambda)\times J(\nu)$. Let $K_{\lambda,\nu}$ be the function on $\overline{N}\times\overline{N}_H$ given by
$$ K_{\lambda,\nu}((x,z),(y,w)) = \NN((-x,-z)\cdot(y,w))^{-2(\nu+\rho_H)}|x_2|^{(\lambda-\rho)+(\nu+\rho_H)}, $$
where
$$ \NN(x,z)=(|x|^4+4|z|^2)^{\frac{1}{4}} $$
and $x=(x_1,x_2)\in\F^{m-1}\oplus\F^{n-m}=\F^{n-1}$. Then the operator $A_{\lambda,\nu}:I(\lambda)\to J(\nu)$ given by
$$ A_{\lambda,\nu}u(y,w) = \int_{\overline{N}} K_{\lambda,\nu}((x,z),(y,w))u(x,z)\,d(x,z) $$
is intertwining for the action of $H$ and depends meromorphically on $(\lambda,\nu)\in\C^2$. For $(\lambda+\rho)+(\pm\nu-\rho_H)\notin(-2\N)$ it was shown in \cite[Theorem 4.1]{MOO16} that $\dim\Hom_H(\pi_\lambda|_H,\tau_\nu)=1$ and the space of intertwining operators is spanned by $A_{\lambda,\nu}$ or its regularization.

In view of the isomorphism $\Hom_H\big(\pi_\lambda|_H,\tau_\nu\big)\cong \Hom_H\big(\pi_\lambda|_H\otimes\tau_{-\nu},\C\big)$, the bilinear form $\ell^\mod_{\lambda,\nu}:I(\lambda)\times J(\nu)\to\C$ given by
\begin{align*}
 \ell^\mod_{\lambda,\nu}(u,v) &= \int_{\overline{N}_H} A_{\lambda,-\nu}u(y,w)\cdot v(y,w)\,d(y,w)\\
 &= \int_{\overline{N}}\int_{\overline{N}_H}K_{\lambda,-\nu}((x,z),(y,w))u(x,z)v(y,w)\,d(y,w)\,d(x,z),
\end{align*}
is invariant for the diagonal action of $H$ by $\pi_\lambda|_H\otimes\tau_\nu$. Further, the uniqueness statement for intertwining operators implies the following uniqueness statement for invariant bilinear forms:

\begin{thm}\label{thm:UniquenessInvBilinearForms}
Let $(G,H)=(\upU(1,n;\F),\upU(1,m;\F)\times\upU(n-m;\F))$, then the space $\Hom_H(\pi_\lambda|_H\otimes\tau_\nu,\C)$ of invariant bilinear forms is one-dimensional if $(\lambda+\rho)+(\pm\nu-\rho_H)\notin(-2\N)$ and spanned by $\ell^\mod_{\lambda,\nu}$ or its regularization.
\end{thm}

In particular, $H$-invariant bilinear forms are unique if $\lambda,\nu\in i\R$. In this case, the integral converges absolutely and no regularization is necessary.

\subsection{Proportionality}

Let $f\in L^2(X)$ and $g\in L^2(Y)$ be Laplace eigenfunctions and let $\pi_f$ and $\tau_g$ be the corresponding irreducible representations of $G$ and $H$. We fix equivariant isometric isomorphisms
$$ \theta:I(\lambda)\to V_f,\,u\mapsto\theta_u \qquad \mbox{and} \qquad \eta:J(\nu)\to W_g,\,v\mapsto\eta_v $$
which map the spherical vectors $\phi_\lambda$ and $\psi_\nu$ to $f$ and $g$. Assuming that $\lambda,\nu\in i\R$, both $\ell^\aut_{f,g}$ and $\ell^\mod_{\lambda,\nu}$ are $H$-invariant bilinear forms on $\pi_f\otimes \tau_g\simeq\pi_\lambda\otimes\tau_\nu$, so that Theorem~\ref{thm:UniquenessInvBilinearForms} implies the existence of a proportionality constant $b_{f,g}\in\C$ such that
$$ \ell^\aut_{f,g}(\theta_u,\eta_v) = b_{f,g}\cdot\ell^\mod_{\lambda,\nu}(u,v) \qquad \forall\,u\in I(\lambda),~v\in J(\nu). $$
In particular,
\begin{equation}
 P_Y(f,g) = \ell^\aut_{f,g}(f,g) = b_{f,g}\cdot\ell^\mod_{\lambda,\nu}(\phi_\lambda,\psi_\nu).\label{eq:Proportionality}
\end{equation}
In the Section~\ref{special} we compute explicitly the special value $\ell^\mod_{\lambda,\nu}(\phi_\lambda,\psi_\nu)$. In Section~\ref{estimate} we derive estimates for the proportionality constants $b_{f,g}$. Together these work implies Theorem~\ref{thm:Main}.

\section{Special values of model invariant bilinear forms}\label{special}

In this section we compute the special value $\ell^\mod_{\lambda,\nu}(\phi_\lambda,\psi_\nu)$ of the model invariant bilinear form $\ell^\mod_{\lambda,\nu}$ at the spherical vectors $\phi_\lambda\in I(\lambda)$, $\psi_\nu\in J(\nu)$.

\begin{thm}\label{thm:SpecialValue}
There exists a constant $c>0$ depending only on $m$, $n$ and $d$ such that
$$ \ell^\mod_{\lambda,\nu}(\phi_\lambda,\psi_\nu) = c\cdot\frac{\Ga(\frac{\lambda+\rho+1}{2})\Ga(\frac{\lambda+\rho-\nu-\rho_H}{2})\Ga(\frac{\lambda+\rho+\nu-\rho_H}{2})}{\Ga(\frac{\lambda+\rho-d+2}{2})\Ga(\lambda+\rho)}. $$
\end{thm}

\begin{rmk}
For $\F=\R$ the above integral was first computed in \cite[Proposition 3.1]{MO13} (see also \cite{KS15} for the special case where $m=n-1$).
\end{rmk}

From the well-known Stirling formula
$$ \Ga(a+ib) = \sqrt{2\pi}|b|^{a-1/2}e^{-\frac{\pi}{2}|b|}(1+\OO(|b|^{-1})), \qquad \text{as}~|b|\to\infty, $$
it follows that for fixed $a>0$ there exist $c_1,c_2>0$ such that
$$ c_1(1+|b|)^{a-1/2}e^{-\frac{\pi}{2}|b|} \leq |\Ga(a+ib)| \leq c_2(1+|b|)^{a-1/2}e^{-\frac{\pi}{2}|b|} \qquad \forall\,b\in\R. $$
Applied to Theorem~\ref{thm:SpecialValue} this yields the following asymptotics for $\ell^\mod_{\lambda,\nu}(\phi_\lambda,\psi_\nu)$:

\begin{cor}\label{cor:AsymptoticsModelPeriod}
For $\lambda,\nu\in i\R$ we have
$$ |\ell^\mod_{\lambda,\nu}(\phi_\lambda,\psi_\nu)| \sim (1+|\lambda|)^{1-n\frac{d}{2}}(1+|\lambda^2-\nu^2|)^{(n-m)\frac{d}{4}-\frac{1}{2}}e^{\frac{\pi}{2}|\lambda|}e^{-\frac{\pi}{4}(|\lambda+\nu|+|\lambda-\nu|)}. $$
\end{cor}

\subsection{Reformulation using intertwining operators}

As noted in Section~\ref{sec:InvBilinearForms} we have
$$ \ell^\mod_{\lambda,\nu}(\phi_\lambda,\psi_\nu) = \int_{\overline{N}_H} A_{\lambda,-\nu}\phi_\lambda(y,w)\cdot\psi_\nu(y,w)\,d(y,w). $$
The intertwining operator $A_{\lambda,-\nu}:I(\lambda)\to J(-\nu)$ maps the $K$-invariant vector $\phi_\lambda\in I(\lambda)$ to a scalar multiple of the $K_H$-invariant vector $\psi_{-\nu}\in J(-\nu)$, and since $\psi_{-\nu}(0,0)=1$ we have
$$ A_{\lambda,-\nu}\phi_\lambda = A_{\lambda,-\nu}\phi_\lambda(0,0)\cdot\psi_{-\nu}. $$
Hence,
\begin{align*}
 \ell^\mod_{\lambda,\nu}(\phi_\lambda,\psi_\nu) &= A_{\lambda,-\nu}\phi_\lambda(0,0)\cdot\int_{\overline{N}_H} \psi_\nu(y,w)\cdot\psi_{-\nu}(y,w)\,d(y,w)\\
 &= A_{\lambda,-\nu}\phi_\lambda(0,0)\cdot\int_{\overline{N}_H} ((1+|y|^2)^2+4|w|^2)^{-\rho_H}\,d(y,w),
\end{align*}
the latter integral being a positive constant depending only on $m$, $n$ and $d$.

\subsection{Computation of $A_{\lambda,-\nu}\phi_\lambda(0,0)$}

We now compute the integral
$$ A_{\lambda,-\nu}\phi_\lambda(0,0) = \int_{\overline{N}} (|x|^4+4|z|^2)^{\frac{\nu-\rho_H}{2}}|x_2|^{(\lambda-\rho)-(\nu-\rho_H)}((1+|x|^2)^2+4|z|^2)^{-\frac{\lambda+\rho}{2}}\,d(x,z). $$
To simplify notation we abbreviate
$$ p=(n-1)d, \qquad p'=(m-1)d, \qquad p''=(n-m)d, \qquad q=d-1, $$
then the integral becomes
\begin{multline*}
 = \int_{\R^{p'}}\int_{\R^{p''}}\int_{\R^q} ((|x_1|^2+|x_2|^2)^2+4|z|^2)^{\frac{2\nu-p'-2q}{4}} |x_2|^{\lambda-\nu-\frac{p''}{2}}\\
 ((1+|x_1|^2+|x_2|^2)^2+4|z|^2)^{-\frac{2\lambda+p+2q}{4}}\,dz\,dx_2\,dx_1.
\end{multline*}
Using polar coordinates on $\R^{p'}$, $\R^{p''}$ and $\R^q$ as well as the volume formula $\vol(S^{k-1})=\frac{2\pi^{\frac{k}{2}}}{\Gamma(\frac{k}{2})}$ this simplifies to
\begin{multline*}
 = \frac{\pi^{\frac{p+q}{2}}}{2^{q-3}\Gamma(\frac{p'}{2})\Gamma(\frac{p''}{2})\Gamma(\frac{q}{2})}\int_0^\infty\int_0^\infty\int_0^\infty((r^2+s^2)^2+t^2)^{\frac{2\nu-p'-2q}{4}}\\
 ((1+r^2+s^2)^2+t^2)^{-\frac{2\lambda+p+2q}{4}} r^{p'-1}s^{\lambda-\nu+\frac{p''}{2}-1}t^{q-1}\,dr\,ds\,dt.
\end{multline*}
With \eqref{eq:IntFormulaMellinTrafoOfSphericalProduct} we first calculate the integral over $t$:
\begin{multline*}
 = \frac{\pi^{\frac{p+q}{2}}\Gamma(\frac{2\lambda-2\nu+2p'+p''+2q}{4})}{2^{q-2}\Gamma(\frac{p'}{2})\Gamma(\frac{p''}{2})\Gamma(\frac{2\lambda-2\nu+2p'+p''+4q}{4})} \int_0^\infty\int_0^\infty r^{p'-1}s^{\lambda-\nu+\frac{p''}{2}-1}(r^2+s^2)^{\nu-\frac{p'}{2}-q}(1+r^2+s^2)^{-\lambda-\frac{p}{2}}\\
 \times {_2F_1}\left(-\tfrac{2\nu-p'-2q}{4},\tfrac{q}{2};\tfrac{2\lambda-2\nu+2p'+p''+4q}{4};1-\tfrac{(1+r^2+s^2)^2}{(r^2+s^2)^2}\right)\,dr\,ds.
\end{multline*}
Substituting $(r,s)=(x\cos\theta,x\sin\theta)$, $x>0$, $0<\theta<\frac{\pi}{2}$ gives
\begin{multline*}
 = \frac{\pi^{\frac{p+q}{2}}\Gamma(\frac{2\lambda-2\nu+2p'+p''+2q}{4})}{2^{q-2}\Gamma(\frac{p'}{2})\Gamma(\frac{p''}{2})\Gamma(\frac{2\lambda-2\nu+2p'+p''+4q}{4})} \int_0^\frac{\pi}{2} \cos^{p'-1}\theta \sin^{\lambda-\nu+\frac{p''}{2}-1}\theta\,d\theta\\
 \times \int_0^\infty x^{\lambda+\nu+\frac{p''}{2}-2q-1} (1+x^2)^{-\lambda-\frac{p}{2}} {_2F_1}\left(-\tfrac{2\nu-p'-2q}{4},\tfrac{q}{2};\tfrac{2\lambda-2\nu+2p'+p''+4q}{4};1-\tfrac{(1+x^2)^2}{x^4}\right)\,dx.
\end{multline*}
Evaluating the first integral using \eqref{eq:SineCosineIntegral} and using the Euler transformation formula \eqref{eq:EulerTransformationFormula} on the hypergeometric function we obtain
\begin{multline*}
 = \frac{\pi^{\frac{p+q}{2}}\Gamma(\frac{2\lambda-2\nu+2p'+p''+2q}{4})\Gamma(\frac{2\lambda-2\nu+p''}{4})}{2^{q-1}\Gamma(\frac{p''}{2})\Gamma(\frac{2\lambda-2\nu+2p'+p''+4q}{4})\Gamma(\frac{2\lambda-2\nu+2p'+p''}{4})} \int_0^\infty x^{-\lambda+\nu-p'-\frac{p''}{2}-2q-1}\\
 \times {_2F_1}\left(\tfrac{2\lambda+p+2q}{4},\tfrac{2\lambda-2\nu+2p'+p''+2q}{4};\tfrac{2\lambda-2\nu+2p'+p''+4q}{4};1-\tfrac{(1+x^2)^2}{x^4}\right)\,dx.
\end{multline*}
Substituting $y=\frac{1+x^2}{x^2}$ we obtain
\begin{multline*}
 = \frac{\pi^{\frac{p+q}{2}}\Gamma(\frac{2\lambda-2\nu+2p'+p''+2q}{4})\Gamma(\frac{2\lambda-2\nu+p''}{4})}{2^q\Gamma(\frac{p''}{2})\Gamma(\frac{2\lambda-2\nu+2p'+p''+4q}{4})\Gamma(\frac{2\lambda-2\nu+2p'+p''}{4})} \int_1^\infty (y-1)^{\frac{2\lambda-2\nu+2p'+p''+4q-4}{4}}\\
 \times {_2F_1}\left(\tfrac{2\lambda+p+2q}{4},\tfrac{2\lambda-2\nu+2p'+p''+2q}{4};\tfrac{2\lambda-2\nu+2p'+p''+4q}{4};1-y^2\right)\,dy.
\end{multline*}
Using the Euler integral representation \eqref{eq:EulerIntegralRepresentation} we further find
\begin{multline*}
 = \frac{\pi^{\frac{p+q}{2}}\Gamma(\frac{2\lambda-2\nu+p''}{4})}{2^q\Gamma(\frac{p''}{2})\Gamma(\frac{q}{2})\Gamma(\frac{2\lambda-2\nu+2p'+p''}{4})} \int_0^\infty t^{\frac{2\lambda-2\nu+2p'+p''+2q-4}{4}}(1+t)^{\frac{2\nu-p'-2q}{4}}\\
 \times\int_1^\infty (y-1)^{\frac{2\lambda-2\nu+2p'+p''+4q-4}{4}} (1+ty^2)^{-\frac{2\lambda+p+2q}{4}}\,dy\,dt
\end{multline*}
The inner integral can be calculated using \eqref{eq:IntRepresentation3F2}:
\begin{multline*}
 = \frac{\pi^{\frac{p+q}{2}}\Gamma(\frac{2\lambda-2\nu+p''}{4})\Gamma(\frac{2\lambda-2\nu+2p'+p''+4q}{4})\Gamma(\frac{2\lambda+2\nu+p''}{4})}{2^q\Gamma(\frac{p''}{2})\Gamma(\frac{q}{2})\Gamma(\frac{2\lambda-2\nu+2p'+p''}{4})\Gamma(\frac{2\lambda+p+2q}{2})} \int_0^\infty t^{-\frac{2\nu-p'+4}{4}}(1+t)^{\frac{2\nu-p'-2q}{4}}\\
 \times{_2F_1}\left(\tfrac{2\lambda+2\nu+p''}{8},\tfrac{2\lambda+2\nu+p''+4}{8};\tfrac{2\lambda+p+2q+2}{4};-t^{-1}\right)\,dt.
\end{multline*}
Substituting $x=t^{-1}$ we find
\begin{multline*}
 = \frac{\pi^{\frac{p+q}{2}}\Gamma(\frac{2\lambda-2\nu+p''}{4})\Gamma(\frac{2\lambda-2\nu+2p'+p''+4q}{4})\Gamma(\frac{2\lambda+2\nu+p''}{4})}{2^q\Gamma(\frac{p''}{2})\Gamma(\frac{q}{2})\Gamma(\frac{2\lambda-2\nu+2p'+p''}{4})\Gamma(\frac{2\lambda+p+2q}{2})} \int_0^\infty x^{\frac{q-2}{2}}(1+x)^{\frac{2\nu-p'-2q}{4}}\\
 \times{_2F_1}\left(\tfrac{2\lambda+2\nu+p''}{8},\tfrac{2\lambda+2\nu+p''+4}{8};\tfrac{2\lambda+p+2q+2}{4};-x\right)\,dx.
\end{multline*}
Finally the last integral can be calculated using \eqref{eq:IntegralHypergeometric}:
\begin{multline*}
 = \frac{\pi^{\frac{p+q}{2}}\Gamma(\frac{2\lambda-2\nu+p''}{4})\Gamma(\frac{2\lambda-2\nu+2p'+p''+4q}{4})\Gamma(\frac{2\lambda+2\nu+p''}{4})\Gamma(\frac{2\lambda-2\nu+2p'+p''}{8})}{2^q\Gamma(\frac{p''}{2})\Gamma(\frac{2\lambda-2\nu+2p'+p''}{4})\Gamma(\frac{2\lambda+p+2q}{2})\Gamma(\frac{2\lambda-2\nu+2p'+p''+4q}{8})}\\
 \times{_3F_2}\left(\tfrac{2\lambda+2\nu+p''}{8},\tfrac{2\lambda-2\nu+2p'+p''+4q}{8},\tfrac{q}{2};\tfrac{2\lambda+p+2q+2}{4},\tfrac{2\lambda-2\nu+2p'+p''+4q}{8};1\right).
\end{multline*}
Simplifying the hypergeometric function ${_3F_2}$ to ${_2F_1}$, evaluating it with \eqref{eq:SpecialValueHypergeometric} and simplifying the whole expression using the duplication formula $\Gamma(z)\Gamma(z+\frac{1}{2})=2^{1-2z}\sqrt{\pi}\Gamma(2z)$ we get
\begin{equation*}
 = \frac{\pi^{\frac{p+q}{2}}\Gamma(\frac{2\lambda-2\nu+p''}{4})\Gamma(\frac{2\lambda+2\nu+p''}{4})\Gamma(\frac{2\lambda+p+2q+2}{4})}{\Gamma(\frac{p''}{2})\Gamma(\frac{2\lambda+p+2q}{2})\Gamma(\frac{2\lambda+p+2}{4})}
\end{equation*}
which is the desired identity.

\section{Estimating the proportionality constants}\label{estimate}

In this section we estimate the size of the scalar $b_{f,g}$ relating the automorphic and the model invariant bilinear form. The technique we use is due to Bernstein--Reznikov~\cite{BR04} and was applied in two other situations in \cite{MO13,MS17}. We therefore omit the details and only point out, at which steps one has to modify the arguments slightly.

\subsection{Hermitian forms and construction of test functions}

For a Maass form $f\in L^2(X)$ with Langlands parameter $\lambda\in i\R$ let
$$ H^\aut_f(\phi) = \int_{\Gamma_H\backslash H} |\phi(h)|^2\,dh, \qquad \phi\in V_f, $$
then $H^\aut_f$ is an $H$-invariant Hermitian form on $V_f$. As explained in\cite{BR04,MO13,MS17}, for an orthonormal sequence $\{g_j\}_j\subseteq L^2(Y)$ of Maass forms on $Y$ with Langlands parameters $\nu_j\in i\R$ the following inequality holds:
\begin{equation}
 \sum_j |b_{f,g_j}|^2 H^\mod_{\lambda,\nu_j}(u) \leq H^\aut(\theta_u) \qquad \forall\,u\in I(\lambda),\label{eq:InequalityHermitianForms}
\end{equation}
where the $H$-invariant Hermitian form $H^\mod_{\lambda,\nu}$ on $I(\lambda)$ is given by
$$ H^\mod_{\lambda,\nu}(u) = \|A^\mod_{\lambda,-\nu}u\|_{L^2(\F^{m-1}\oplus\Im\F)}^2, \qquad u\in I(\lambda). $$

\begin{lem}\label{lem}
There exist $C_1,C_2>0$ and $\xi\in C_c^\infty(G)$, $\xi\geq0$, such that for $T\gg0$ and all Maass forms $f\in L^2(X)$ and $g\in L^2(Y)$ with Langlands parameters $\lambda,\nu\in i\R$, $|\lambda|,|\nu|\leq T$, there exists a test function $u_T\in C_c^\infty(\F^{n-1}\oplus\Im\F)\subseteq I(\lambda)$ of $L^2$-norm one such that
\begin{enumerate}[(1)]
\item $\int_G H^\mod_{\la,\nu}(\pi_\lambda(k)u_T)\xi(k)\,dk \geq C_1\cdot T^{-2\rho}$,
\item $\int_G H^\aut_f(\theta_{\pi_\lambda(k)u_T})\xi(k)\,dk \leq C_2$.
\end{enumerate}
\end{lem}

\begin{proof}
The same statement in the context of $(G,H)=(\GL(3,\R),\GL(2,\R))$ is proven in \cite[Lemma 4.1]{MS17}, and we only provide the necessary details to translate the proof to our setting. First, we note that $K_{\lambda,\nu}((x,z),(0,0))=1$ for $(x,z)=(e_{n-1},0)\in\F^{n-1}\oplus\Im\F$. Let $u\in C_c^\infty(\F^{n-1}\oplus\Im\F)$, $u\geq0$, with support in a sufficiently small neighborhood around $(0,0)$, normalized to have $L^2$-norm one. Then
$$ u_T(x,z) = T^{\frac{p+2q}{2}}u(T(x-e_{n-1}),T^2z), \qquad (x,z)\in\F^{n-1}\oplus\Im\F, $$
defines a function $u_T\in C_c^\infty(\F^{n-1}\oplus\Im\F)\subseteq I(\lambda)$ of $L^2$-norm one with support around $(e_{n-1},0)$. As in \cite[Lemma 4.1]{MS17} one shows that $(\pi_\lambda(k)K_{\lambda,\nu})((x,z),(0,0))\sim1$ whenever $(x,z)\in\supp u_T$, $|\lambda|,|\nu|\leq T$ and $k$ is contained in some identity neighborhood $U\subseteq G$. This implies
$$ A_{\lambda,-\nu}\pi(k)u_T(0,0) \sim \int_{\F^{n-1}\oplus\Im\F} u_T(x,z)\,d(x,z) = \const\cdot T^{-\rho}  $$
for $k\in U$, $|\lambda|,|\nu|\leq T$ and $T\gg0$. Then for any left-$K_H$-invariant function $\xi\in C_c^\infty$, $\xi\geq0$, with $\xi|_U\equiv1$ we obtain
$$ \int_G H^\mod_{\lambda,\nu}(\pi_\lambda(k)u_T)\xi(k)\,dk = \int_G |A_{\lambda,-\nu}^\mod\pi_\lambda(k)u_T(0,0)|^2\xi(k)\,dk \geq C_1\cdot T^{-2\rho}, $$
which shows (1). For (2) the argument is the same as in \cite[Lemma 4.1]{MS17}.
\end{proof}

\begin{prop}\label{thm-1}
There exists a constant $C_3>0$ such that
$$ \sum_{|\nu_j|\leq T}|b_{f,g_j}|^2 \leq C_3\cdot T^{2\rho} \qquad \forall\,f\text{ with }|\lambda|\leq T. $$
\end{prop}

\begin{proof}
Apply Lemma~\ref{lem} to the inequality \eqref{eq:InequalityHermitianForms}, then the claim follows (see \cite[Proposition 4.2]{MS17} for details).
\end{proof}

\subsection{Proof of the main results}

Theorem~\ref{thm:Main} now follows from \eqref{eq:Proportionality} together with Corollary~\ref{cor:AsymptoticsModelPeriod} and Proposition~\ref{thm-1} after replacing the Langlands parameters $\lambda$ and $\nu$ by the eigenvalues $a=\rho^2-\lambda^2$ and $b=\rho_H^2-\nu^2$. Corollary~\ref{cor:Main} easily follows from Theorem~\ref{thm:Main} by fixing one of the Maass forms.

\appendix

%% The following typesetting is for designing the numbers of sections and equations in appendix. It should be removed when adapting the script to journal TeX templates.  
\titleformat{\section}{\sc\center}{}{0.5em}{} 
\renewcommand{\theequation}{A.\arabic{equation}}
\setcounter{equation}{0}
%%%%%%%%%%%%%%%%%%%%

\section{Appendix A. Integral formulas}% remove ''Appendix A.'' when adapting the script to journal TeX templates. 

For $\alpha,\beta>0$ and $0<\Re\lambda<2\Re(\mu+\nu)$ we have (see \cite[equation 3.259~(3)]{GR07})
\begin{equation}
 \int_0^\infty x^{\lambda-1} (1+\alpha x^2)^{-\mu} (1+\beta x^2)^{-\nu}\,dx = \frac{1}{2}\alpha^{-\frac{\lambda}{2}}B\left(\frac{\lambda}{2},\mu+\nu-\frac{\lambda}{2}\right){_2F_1}\left(\nu,\frac{\lambda}{2};\mu+\nu;1-\frac{\beta}{\alpha}\right).\label{eq:IntFormulaMellinTrafoOfSphericalProduct}
\end{equation}
The Euler transformation formula holds (see \cite[equation (2.2.7)]{AAR99}):
\begin{equation}
 {_2F_1}(\alpha,\beta;\gamma;x) = (1-x)^{\gamma-\alpha-\beta}{_2F_1}(\gamma-\alpha,\gamma-\beta;\gamma;x).\label{eq:EulerTransformationFormula}
\end{equation}
The Euler integral representation holds for $\Re(\gamma-\beta),\Re\beta>0$ (see \cite[equation (2.3.17)]{AAR99}):
\begin{equation}
 {_2F_1}(\alpha,\beta;\gamma;1-x) = \frac{\Gamma(\gamma)}{\Gamma(\gamma-\beta)\Gamma(\beta)}\int_0^\infty t^{\beta-1}(1+t)^{\alpha-\gamma}(1+xt)^{-\alpha} dt.\label{eq:EulerIntegralRepresentation}
\end{equation}
The following integral formula holds for $0\leq\beta<1$ and $0<\Re\mu<\Re(\lambda-2\nu)$ (see \cite[equation 3.254~(2)]{GR07} for $\lambda=0$ and $u=1$):
\begin{equation}
 \int_1^\infty (x-1)^{\mu-1}(x^2+\beta)^\nu dx = \frac{\Gamma(\mu)\Gamma(-\mu-2\nu)}{\Gamma(-2\nu)}{_2F_1}\left(-\frac{\mu}{2}-\nu,\frac{1-\mu}{2}-\nu;\frac{1}{2}-\nu;-\beta\right).\label{eq:IntRepresentation3F2}
\end{equation}
Using Pfaff's transformation formula (see \cite[equation (2.2.6)]{AAR99})
\begin{equation*}
 {_2F_1}(\alpha,\beta;\gamma;x) = (1-x)^{-\alpha} {_2F_1}\left(\alpha,\gamma-\beta;\gamma;\frac{x}{x-1}\right)
\end{equation*}
and the following integral formula which holds for $\Re\rho,\Re\sigma,\Re(\gamma+\sigma-\alpha-\beta)>0$ (see \cite[equation 7.512~(5)]{GR07})
\begin{equation*}
 \int_0^1 x^{\rho-1}(1-x)^{\sigma-1}{_2F_1}(\alpha,\beta;\gamma;x) dx = \frac{\Gamma(\rho)\Gamma(\sigma)}{\Gamma(\rho+\sigma)}{_3F_2}(\alpha,\beta,\rho;\gamma,\rho+\sigma;1)
\end{equation*}
it is easy to see by a simple substitution that for $\Re\rho,\Re(\alpha-\sigma-\rho+1),\Re(\beta-\sigma-\rho+1)>0$ we have
\begin{equation}
 \int_0^\infty x^{\rho-1}(1+x)^{\sigma-1}{_2F_1}(\alpha,\beta;\gamma;-x) dx = \frac{\Gamma(\rho)\Gamma(\alpha-\sigma-\rho+1)}{\Gamma(\alpha-\sigma+1)}{_3F_2}(\alpha,\gamma-\beta,\rho;\gamma,\alpha-\sigma+1;1).\label{eq:IntegralHypergeometric}
\end{equation}
For $\Re(\gamma-\alpha-\beta)>0$ we have the {Gau\ss} special value (see \cite[Theorem 2.2.2]{AAR99})
\begin{equation}
 {_2F_1}(\alpha,\beta;\gamma;1) = \frac{\Gamma(\gamma)\Gamma(\gamma-\alpha-\beta)}{\Gamma(\gamma-\alpha)\Gamma(\gamma-\beta)}.\label{eq:SpecialValueHypergeometric}
\end{equation}
For $\Re\mu,\Re\nu>0$ we have (see \cite[equation 3.621~(5)]{GR07})
\begin{equation}
 \int_0^{\frac{\pi}{2}} \sin^{\mu-1}\theta \cos^{\nu-1}\theta d\theta = \frac{1}{2}B\left(\frac{\mu}{2},\frac{\nu}{2}\right).\label{eq:SineCosineIntegral}
\end{equation}

%------- References -------
\providecommand{\bysame}{\leavevmode\hbox to3em{\hrulefill}\thinspace}
\providecommand{\MR}{\relax\ifhmode\unskip\space\fi MR }
\providecommand{\MRhref}[2]{%
  \href{http://www.ams.org/mathscinet-getitem?mr=#1}{#2}
}
\providecommand{\href}[2]{#2}

\begin{flushleft}\small
Department Mathematik, FAU Erlangen-N\"{u}rnberg\\
Cauerstr. 11, 91058 Erlangen, Germany\\
\texttt{moellers@math.fau.de}

Institute of Mathematics, AMSS, Chinese Academy of Sciences\\
Zhongguancun East Road 55, 100190 Beijing, China\\
\texttt{fsu@amss.ac.cn}
\end{flushleft}

\end{document}